\documentclass{jgcc} 

\pdfoutput=1
\usepackage{lastpage}
\jgccdoi{15}{1}{1}{12200}  
\jgccheading{}{\pageref{LastPage}}{}{}{Feb.~4,~2022}{Aug.~30,~2023}{}    



\usepackage{hyperref}
\theoremstyle{plain} 

\usepackage{amssymb}
\usepackage{amsmath}
\usepackage{amsfonts}

\def\dprod{\mathop{\displaystyle \prod }}%
\def\dbigcap{\mathop{\displaystyle \bigcap }}%
\def\dbigwedge{\mathop{\displaystyle \bigwedge }}%
\def\dcoprod{\mathop{\displaystyle \coprod }}%
\def\func#1{\mathop{\rm #1}}%

\begin{document}

\title[Heisenberg Group]{An Axiomatization for the Universal Theory \\of the
Heisenberg Group}

\author[A.~M.~Gaglione]{Anthony M. Gaglione}
\address{Departmant of Mathematics\\
U.S. Naval Academy\\
Annapolis, MD 21402}
\email{agagaglione@aol.com}

\author[D.~Spellman]{Dennis Spellman}
\address{Department of Statistics\\
Temple University\\
Philadelphia, PA 19122}
\email{dennisspellman1@aol.com}

\dedicatory{This paper is dedicated to the memory of Benjamin Fine.}

\begin{abstract}
The Heisenberg group, here denoted $H$, is the group of all $3\times 3$
upper unitriangular matrices with entries in the ring $\mathbb{Z}$ of
integers. A.G. Myasnikov posed the question of whether or not the universal
theory of $H$, in the language of $H$, is axiomatized, when the models are
restricted to $H$-groups, by the quasi-identities true in $H$ together with
the assertion that the centralizers of noncentral elements be abelian. Based
on earlier published partial results we here give a complete proof of a
slightly stronger result.
\end{abstract}

\maketitle

\section{Introduction}

A (multiplicatively written) group $G$ is 
\emph{commutative transitive}, briefly CT, provided the relation of
commutativity is transitive on $G\backslash \{1\}$; equivalently, provided
the centralizer of every element $g\neq 1$ is abelian.


Noncyclic free groups are universally equivalent, even
elementarily equivalent. Myasnikov and Remeslennikov [MR] proved that their
universal theory is axiomatized by the quasi-identities they satisfy
together with commutative transitivity. Fixing a noncyclic free group $F$,
they proved the analogous result in the language of $F$ when the models
are restricted to $F$-groups.


Let $A$ be a countably infinite set well-ordered as%
\begin{equation*}
\{a_{1},a_{2},...,a_{n},...\}=\{a_{n+1}:~n<\omega \}
\end{equation*}%
where $\omega $ is the first limit ordinal which we take as the set of
nonnegative integers provided with its natural order. Let $F_{\omega }(%
\mathcal{N}_{2})$ be the group free in the variety of all 2-nilpotent groups
on the generators $A$. For each integer $n\geq 2$ let $F_{n}(\mathcal{N}%
_{2}) $ be the subgroup of $F_{\omega }(\mathcal{N}_{2})$ generated
(necessarily freely) by the initial segment $\{a_{1},a_{2},...,a_{n}\}$ of $%
A $. The \emph{Heisenberg group} is the group $H$ of all $3\times 3$
upper unitriangular matrices with entries in the ring $\mathbb{Z}$ of
integers. It is free in $\mathcal{N}_{2}$ on the generators%
\begin{equation*}
a_{1}=%
\begin{bmatrix}
1 & 0 & 0 \\ 
0 & 1 & 1 \\ 
0 & 0 & 1%
\end{bmatrix}%
\ \ \text{and} \ \ a_{2}=%
\begin{bmatrix}
1 & 1 & 0 \\ 
0 & 1 & 0 \\ 
0 & 0 & 1%
\end{bmatrix}%
\end{equation*}%
(See \cite{GS2}). We take the liberty of identifying $F_{2}(\mathcal{N}_{2})$
with $H$.


Now $F_{\omega }(\mathcal{N}_{2})$ is \emph{discriminated} by
the family of retractions $F_{\omega }(\mathcal{N}_{2})\rightarrow H$. That
means that given finitely many elements $f_{1},...,f_{k}\in F_{\omega }(\mathcal{N}_{2})\backslash \{1\}$ there is a retraction $F_{\omega }(\mathcal{N}_{2})\rightarrow H$ which doesn't annihilate any of them. ($H$ 
\emph{discriminates } $\mathcal{N}_{2}$ in the sense of Hanna Neumann
\cite{N}.) From this it follows that $F_{n}(\mathcal{N}_{2})$, $n\geq 2$, are
universally equivalent; moreover, since the discrimination is done by
retractions, they are universally equivalent in the language of $H$. (See
\cite{GS1}).

 
 Let CT(1) or \emph{noncentral commutative transitivity}, briefly
NZCT, be the property that the relation of commutativity be transitive on $G\backslash Z(G)$ where $Z(G)$ is the center of $G$. Equivalently, NZCT
asserts that the centralizers of noncentral elements are abelian. A special
case of a question posed by A.G. Myasnikov is whether or not the universal
theory of noncyclic free 2-nilpotent groups is axiomatized by the
quasi-identities they satisfy together with NZCT and whether or not that
theory in the language of $H~$\ is so axiomatized when the models are
restricted to $H$-groups.

 
 In this paper, in the case of the language of $H$, we answer that
question in the positive. In fact we prove a slightly stronger result. The
remainder of this paper contains four \ additional sections. In Section 2
we fix definitions and notation. In Section 3 
we prove the main result. In
Section 4 
we ponder but do not settle the
question in the language without parameters from $H$. Finally in Section 5
we suggest problems for future research.

 
 Before closing the introduction we note that the variety $\mathcal{A}^{2}$ of metabelian groups is discriminated by its rank $2$ free group. From
this it follows that the noncyclic free metabelian groups are universally
equivalent. It is worth mentioning in passing that Remeslenikov and Stohr
proved in \cite{RS} that the universal theory of the noncyclic free metabelian
groups is axiomatized by the quasi-identities they satisfy together with
commutative transitivity.

\medskip

\section{Pedantic Preliminaries}

\medskip\ Let $L_{0}$ be the first order language with equality containing a
binary operation symbol $\cdot $ , a unary operational symbol $^{-1}$ and a
constant symbol $\widehat{1}.$ If $G$ is a (multiplicatively written) group $L_{0}[G]$ is obtained from $L_{0}$ by adjoining names $\widehat{g}$ for the
elements $g\in G\backslash \{1\}$ as new constant symbols. We find it
convenient to commit the "abuses" of identifying $\widehat{g}$ with $g$ for
all $g\in G$ and replacing $\cdot $ with juxtaposition. Moreover, we find it
convenient to write an inequation $\sim (s=t)$ as $s\neq t$.

 
 An \emph{identity}, in $L_{0}[G]$ \ ({Note}: $L_{0}=L[\{1\}]$) is a universal sentence of the form $\forall \overline{x}~(s(\overline{x})=t(\overline{x}))$ where $\overline{x}$ is a tuple of
variables and $s(\overline{x})$ and $t(\overline{x})$ are terms of $L_{0}[G]$. Examples of identities are the group axioms, namely:%
\begin{eqnarray*}
\forall x_{1},x_{2},x_{3}~((x_{1}x_{2})x_{3}~ &=&~x_{1}(x_{2}x_{3})) \\
\forall x~(x1) &=&x \\
\forall x~(x~x^{-1})\ &=&\ 1).
\end{eqnarray*}%
A \emph{quasi-identity} of $L_{0}[G]$ is universal sentence of the form%
\begin{equation*}
\forall~\overline{x}~\left( \dbigwedge\limits_{i=1}^{k}~(~s_{i}(~\overline{x%
})=t_{i}(~\overline{x}))\rightarrow \ (s(~\overline{x})=t(~\overline{x}%
))\right)
\end{equation*}%
where $~\overline{x}$ is a tuple of variables and the $s_{i}(~\overline{x})$%
, $t_{i}(~\overline{x}),~\ s(~\overline{x})$ and $t(~\overline{x})$ are
terms of $L_{0}[G]$. Every identity is equivalent to a quasi-identity since $%
\forall \overline{x}~(s(\overline{x})=t(\overline{x}))$ is equivalent to 
\begin{equation*}
\forall \overline{x},y~((y=y)\rightarrow ~(s(\overline{x})=t(\overline{x}))).
\end{equation*}%
In particular, the group axioms are equivalent to quasi-identities.

 
 If $G$ is a group and we let $\mathcal{Q}^{0}(G)$ be the set of all
quasi-identities of $L_{0}$ true in $G$ and $\mathcal{Q}(G)$ be the set of
all quasi-identities of $L_{0}[G]$ true in $G$. We view the group axioms as
contained in $\mathcal{Q}^{0}(G)\subseteq \mathcal{Q}(G)$. We let $%
Th_{\forall }^{0}(G)$ be the set of all universal sentences of $L_{0}$ true
in $G$ and $Th_{\forall }(G)$ be the set of all universal sentences of $%
L_{0}[G]$ true in $G$. \ Note that quantifier free sentences are viewed as
special cases of universal sentences. In particular, the \emph{diagram}
of $G$, briefly $\text{diag}(G)$, consisting of the atomic and negated atomic
sentences of $L_{0}[G]$ true in $G$ is a set of universal sentences of $%
L_{0}[G]$.

 
 A \emph{$G$-group} $\Gamma $ is a model of the group axioms and $\text{diag}(G)$. That is equivalent to the group $\Gamma $ containing a
distinguished copy of $G$ as a subgroup. A \emph{$G$-polynomial} is a
group word on the elements of $G$ and variables. (If you like, an element of
the free product $G\ast \langle x_{1},...,x_{n};~\rangle $ for some $n$.)
Note that, modulo the group axioms, every identity of $L_{0}$ ($L_{0}[G]$ )
is equivalent to one of the form $\forall \overline{x}~\left( w(\overline{x}%
)=1\right) $ where $w(\overline{x})$ is a group word ($G$-polynomial) and
every quasi-identity of $L_{0}$ ($L_{0}[G]$ ) is equivalent to one of the
form%
\begin{equation*}
\forall ~\overline{x}~\left( \dbigwedge\limits_{i=1}^{k}~(~u_{i}(~\overline{x%
})=1)\rightarrow \ (w(~\overline{x})=1)\right)
\end{equation*}%
where the $u_{i}(\overline{x})$ and $w(~\overline{x})$ are group words ($G$%
-polynomials).

 
 In this paper by "ring" we shall always mean commutative ring with
multiplicative identity $1\neq 0$. Subrings are required to contain $1$ and
homomorphisms are required to preserve $1$. A ring $R$ is \emph{residually}-$\mathbb{Z}$ provided, given $r\in R\backslash \{0\}$, there is
a homomorphism $R\rightarrow \mathbb{Z}$ which does not annihilate $r$. This
forces $R$ to have characteristic zero and we identify the minimum subring
of $R$ with \ $\mathbb{Z}$.  Hence, we view $R$ as separated by
retractions $R\rightarrow \mathbb{Z}$.  A ring is \emph{locally
residually-$\mathbb{Z}$ } provided every finitely generated subring is
residually-$\mathbb{Z}$ . \ Being locally residually-$\mathbb{Z}$ is
equivalent to being a model of the quasi-identities of ring theory true in $\mathbb{Z}$ (See \cite{GS2}).

 
 It was proven in \cite{GS2} that every model of $\mathcal{Q}(H)\cup
\text{diag}(H) $ $H$-embeds in $UT_{3}(R)$ for some locally residually-$\mathbb{Z}$
ring $R$. (Conversely every $H$-subgroup of such a $UT_{3}(R)$ is a model of
\ $\mathcal{Q}(H)\cup \text{diag}(H)$.) Here an \emph{$H$-embedding} is an
embedding which is the identity on $H$. (The meanings of \emph{$H$-subgroup} and \emph{$H$-homomorphism} in the category of $H$-groups
are readily apparent.)

 
 Let $G$ be a group and $g\in G$. We let $C_{G}(g)$ be the centralizer
of $g$ in $G$. NZCT is the following universal sentence of $L_{0}:$%
\begin{equation*}
\forall x_{1},x_{2},x_{3},y~\left( (([x_{2},y]\neq 1)\wedge
([x_{1},x_{2}]=1)\wedge ([x_{2},x_{3}]=1))\rightarrow
([x_{1},x_{3}]=1)\right) .
\end{equation*}%
A group $G$ satisfies NZCT if and only if $C_{G}(g)$ is abelian for all $g\in G\backslash Z(G)$. The following quasi-identity of $L_{0}[H]$ holds in $H$.%
\begin{equation*}
\forall x,z~\left( (([z,a_{1}]=1)\wedge ([a_{2},z]=1))\rightarrow
([z,x]=1)\right) .
\end{equation*}%
It follows that, if $G$ is any model of $\mathcal{Q}(H)\cup \text{diag}(H)$, then $%
C_{G}(a_{1})\cap C_{G}(a_{2})=Z(G)$. In particular, if $G$ additionally
satisfies NZCT, then the following universal sentence $\tau $ of $L_{0}[H]$
is a consequence%
\begin{equation*}
\forall x_{1},x_{2}~\left( 
\begin{array}{c}
(([x_{2},x_{1}]=1)\wedge ([a_{2},x_{2}]=1)\wedge
([x_{1},a_{1}]=1)~\rightarrow \\ 
(([x_{2},a_{1}]=1)\vee ([a_{2},x_{1}]=1))%
\end{array}%
\right) .
\end{equation*}%
We shall prove in the next section that $\mathcal{Q}(H)\cup \text{diag}(H)\cup
\{\tau \}$ axiomatizes $Th_{\forall }(H)$.

 
 A \emph{primitive sentence} of $L_{0}[H]$ is ($\func{mod}$ulo
the group axioms) an existential sentence of the form $\exists ~\overline{x}%
~\left( \dbigwedge\limits_{i}~(p_{i}(~\overline{x})=1)~\wedge
(~\dbigwedge\limits_{j}(q_{j}(~\overline{x})\neq 1)\right) $ where the $%
p_{i}(~\overline{x})$ and $q_{j}(~\overline{x})$ are $H$-polynomials.

 
 By writing the matrix of an existential sentence in disjunctive
normal form one sees that every existential sentence of $L_{0}[H]$ \ is
equivalent (modulo the group axioms) to a disjunction of primitive sentences
and so holds in an $H$-group if and only if at least one disjunct does.

 
 Assume momentarily that there is a universal sentence $\varphi $ of $%
L_{0}[H]$ which holds in $H$ but is not a consequence of $\mathcal{Q}(H)\cup
\text{diag}(H)\cup \{\tau \}$. Then its negation $\sim \varphi $ must hold in some
model $\ G$ of $\mathcal{Q}(H)\cup \text{diag}(H)\cup \{\tau \}$. Since $\sim
\varphi $ is equivalent to an existential sentence of $L_{0}[H]$ there must
be a primitive sentence 
\begin{equation*}
\exists x_{1},...,x_{k}~\left(
\dbigwedge\limits_{i}~(p_{i}(~a_{1},a_{2},x_{1},...,x_{k})=1)~\wedge
(~\dbigwedge\limits_{j}(q_{j}(a_{1},a_{2},x_{1},...,x_{k})\neq 1)\right)
\end{equation*}%
of $L_{0}[H]$ which holds in $G$ but is false in $H$. Let the assignment $%
x_{\lambda }\mapsto g_{\lambda }$, $1\leq \lambda \leq k$, verify the above
sentence in $G$. Then the finitely generated $H$-subgroup $G_{0}=~\langle
a_{1},a_{2},g_{1},...g_{k}\rangle $ of $G$ also satisfies the above
primitive sentence \ and, since universal sentences of $L_{0}[H]$ are
preserved in $H$-subgroups, $G_{0}$ is a model of $\mathcal{Q}(H)\cup
\text{diag}(H)\cup \{\tau \}$. Hence, if a counterexample exists, then so would a
finitely generated counterexample exist. So it suffices to prove the result
for finitely generated models. We shall find it convenient to prove, more
generally, that the result holds for models $G$ such that the quotient $%
\overline{G}=G/Z(G)$ is finitely generated.

\medskip

\section{The Lame Property and the Universal Theory of $H$}

\medskip

 
 It was shown in \cite{GS4} using a characterization due to Mal'cev \cite{M}
that a model $E$ of $\mathcal{Q}(H)\cup \text{diag}(H)$ is of the form $UT_{3}(R)$
for some locally residually-$\mathbb{Z}$ ring $\ R$ if and only if $E$
satisfies the following universal-existential sentence

$\sigma $ of $L_{0}[H]:$%
\begin{equation*}
\forall x_{1},x_{2}\exists y_{1},y_{2}~\left( 
\begin{array}{c}
([y_{1},a_{1}]=1)\wedge ([a_{2},y_{2}]=1)\wedge ([x_{2},x_{1}]=[y_{2},a_{1}])
\\ 
\wedge ~(~[x_{2},x_{1}]=[a_{2},y_{1}])%
\end{array}%
\right) .
\end{equation*}%
 
 In other words, for each commutator $[g_{2},g_{1}]$, each of the
systems%
\begin{equation*}
S~\left\{ 
\begin{array}{c}
\lbrack a_{2},y]=1 \\ 
\lbrack y,a_{1}]=[g_{2},g_{1}]%
\end{array}%
\right.
\end{equation*}%
and%
\begin{equation*}
T~\left\{ 
\begin{array}{c}
\lbrack x,a_{1}]=1 \\ 
\lbrack a_{2},x]=[g_{2},g_{1}]%
\end{array}%
\right.
\end{equation*}%
has a solution in $E$.

 
 Mimicking the construction of an existentially closed extension (See
e.g. Hodges [H]) we may embed a model $G$ of \ $\mathcal{Q}(H)\cup
\text{diag}(H)\cup \{\tau \}$ into $UT_{3}(R)$, If we could preserve $\tau $ at
each step then, since universal sentences are preserved in direct unions, $%
UT_{3}(R)$ would also satisfy $\tau $. If $UT_{3}(R)$ satisfies $\tau $,
then $R$ is an integral domain. To see this let $(r,s)\in R^{2}$ and suppose 
$rs=0$. Let $y=%
\begin{bmatrix}
1 & r & 0 \\ 
0 & 1 & 0 \\ 
0 & 0 & 1%
\end{bmatrix}%
$ and $x=%
\begin{bmatrix}
1 & 0 & 0 \\ 
0 & 1 & s \\ 
0 & 0 & 1%
\end{bmatrix}%
$. Then $[y,x]=1$ and $[a_{2},y]=1$ and $[x,a_{1}]=1$. Moreover, $[y,a_{1}]=%
\begin{bmatrix}
1 & 0 & r \\ 
0 & 1 & 0 \\ 
0 & 0 & 1%
\end{bmatrix}%
$ and $[a_{2},x]=%
\begin{bmatrix}
1 & 0 & s \\ 
0 & 1 & 0 \\ 
0 & 0 & 1%
\end{bmatrix}%
$. \ Since $UT_{3}(R)$ satisfies $\tau $, $r=0$ or $s=0$ and $R$ indeed is
an integral domain.

 
 A residually-$\mathbb{Z}$ ring $R$ is \emph{$\omega $-residually-%
$\mathbb{Z}$} provided it is discriminated by the family of retractions $%
R\rightarrow \mathbb{Z}$ . That is, given finitely many nonzero elements of $%
R$, there is a retraction $R\rightarrow \mathbb{Z}$ which does not
annihilate any of them. That is equivalent to $R$ being an integral domain.
Suppose first that $R$ is an integral domain. Let $r_{1},...,r_{n}$ be
finitely many nonzero elements of $R$. Then the product $r=$ $r_{1}\cdots
r_{n}\neq 0$ and there is a retraction $\rho :R\rightarrow \mathbb{Z}$ which
does not annihilate $r$ so cannot annihilate and of $r_{1},...,r_{n}$.
Conversely, if $R$ is $\omega $-residually-$\mathbb{Z}$ and $r$ and $s$ are
nonzero elements of $R$, then there is a retraction $\rho :R\rightarrow 
\mathbb{Z}$ such that $\rho (r)\neq 0$ and $\rho (s)\neq 0$; so, $\rho
(rs)=\rho (r)\rho (s)\neq 0$ and hence $rs\neq 0$ and $R$ must be an
integral domain. From this it follows that a locally residually-$\mathbb{Z}$
ring $R$ is locally $\omega $-residaully-$\mathbb{Z}$ if and only if $R$ is
an integral domain. So, if $UT_{3}(R)$ satisfies $\tau $, then $R$ is
locally $\omega $-residually-$\mathbb{Z}$ .

 
 Given $G\leq _{H}UT_{3}(R)$, $G$ is the direct union $\underset{%
\rightarrow }{\lim }~(G\cap UT_{3}(R_{0}))$ as $R_{0}$ varies over the
finitely generated subrings of $R$. Ring retractions $R_{0}\rightarrow 
\mathbb{Z}$ induce group retractions $G\cap UT_{3}(R_{0})\rightarrow H$.
From this it follows that each $G\cap UT_{3}(R_{0})$ is discriminated by the
family of group retractions $G\cap UT_{3}(R_{0})\rightarrow H$ and hence
each $G\cap UT_{3}(R_{0})$ is a model of $Th_{\forall }(H)$. Since universal
sentences are preserved in direct unions and $G=$ $\underset{\rightarrow }{%
\lim }~(G\cap UT_{3}(R_{0}))$ we have that $G$ is a model of $Th_{\forall
}(H)$. The bottom line is that we would be finished if we could construct an
overgroup $UT_{3}(R)$ in such a way that $\tau $ is preserved at each step
of the construction.

 
 Let us forget momentarily about this particular $R$ and consider a
possible property that a representation $G\leq _{H}UT_{3}(R)$ might satisfy.

\smallskip

\begin{defi}
Let $G$ be a model of \ $\mathcal{Q}(H)\cup \text{diag}(H)$ and let $G\leq
_{H}UT_{3}(R)$ where $R$ is locally residually-$\mathbb{Z}$. We say the
representation satisfies the \emph{Lame Property} provided, for each $g=%
\begin{bmatrix}
1 & g_{12} & g_{13} \\ 
0 & 1 & g_{23} \\ 
0 & 0 & 1%
\end{bmatrix}%
\in C_{G}(a_{1})\cup C_{G}(a_{2})$, either $g_{12}^{2}+g_{23}^{2}=0$ or $%
g_{12}^{2}+g_{23}^{2}$ is not a zero divisor in $R$.
\end{defi}

\smallskip

\begin{lem}
Given a model $G$ of $\mathcal{Q}(H)\cup \text{diag}(H)$ and a representation $%
G\leq _{H}UT_{3}(R)$ where $R$ locally residually-$\mathbb{Z}$. The
representation satisfies the Lame Property if and only if it satisfies the
conjunction of the following two conditions:
\end{lem}

\textit{\ \smallskip (1.) \ \ For all }$y=%
\begin{bmatrix}
1 & y_{12} & y_{13} \\ 
0 & 1 & 0 \\ 
0 & 0 & 1%
\end{bmatrix}%
\in C_{G}(a_{2})\backslash Z(G)$\textit{, }$~y_{12}$\textit{\ is not a zero
divisor in }$R$\textit{.}

\textit{(2.) For all }$x=%
\begin{bmatrix}
1 & 0 & x_{13} \\ 
0 & 1 & x_{23} \\ 
0 & 0 & 1%
\end{bmatrix}%
\in C_{G}(a_{1})\backslash Z(G)$\textit{, }$x_{23}$\textit{\ is not a zero
divisor in }$R$\textit{.}

\smallskip

\begin{proof}
Suppose the representation satisfies the Lame Property. Let $y=%
\begin{bmatrix}
1 & y_{12} & y_{13} \\ 
0 & 1 & 0 \\ 
0 & 0 & 1%
\end{bmatrix}%
\in C_{G}(a_{2})\backslash Z(G)$. Then $y_{12}\neq 0.$

 
 The quasi-identity $\forall x~((x^{2}=0)\rightarrow (x=0))$ holds in $%
\mathbb{Z}$ ; hence, it is true in $R$ and $y_{12}^{2}\neq 0$. \ If $r\neq 0$
annihilates $y_{12}$ then $r(y_{12}^{2}+y_{23}^{2})=ry_{12}^{2}=0$
contradicting $y_{12}^{2}=y_{12}^{2}+y_{23}^{2}$ is not a zero divisor in $R$%
. The contradiction shows that the representation satisfies (1.). Similarly,
the representation satisfies (2.).

 
 Now suppose that representation satisfies (1.) and (2.). Let $g=%
\begin{bmatrix}
1 & g_{12} & g_{13} \\ 
0 & 1 & g_{23} \\ 
0 & 0 & 1%
\end{bmatrix}%
\in C_{G}(a_{1})\cup C_{G}(a_{2})$. $\mathbb{Z}$ satisfies each of the
quasi-identities%
\begin{eqnarray*}
\forall x,y~((x^{2}+y^{2} &=&0)\rightarrow (x=0))\text{ and} \\
\forall x,y~(x^{2}+y^{2} &=&0)\rightarrow (y=0));\text{ so,}
\end{eqnarray*}%
they must hold in $R$. Hence, if $g_{12}^{2}+g_{23}^{2}\neq 0$, then either $%
g_{12}\neq 0$ or $g_{23}\neq 0$.

Suppose there were an $r\neq 0$ which annihilates $g_{12}^{2}+g_{23}^{2}$.
From $r(g_{12}^{2}+g_{23}^{2})=0$ we get $%
(rg_{12})^{2}+(rg_{23})^{2}=r^{2}(g_{12}^{2}+g_{23}^{2})=0$. Hence $%
rg_{12}=0 $ and $rg_{23}=0$. If $g_{12}\neq 0$ (1.) is contradicted while if 
$g_{23}\neq 0$ (2.) is contradicted. The contradiction shows that the
conjunction of (1.) and (2,) implies the Lame Property. 
\end{proof}

\smallskip

 
 We next note that if the representation $G\leq _{H}UT_{3}(R)$
satisfies the Lame Property, then $G$ satisfies $\tau $.

 
 For suppose $[a_{2},y]=1$ so $y=%
\begin{bmatrix}
1 & y_{12} & y_{13} \\ 
0 & 1 & 0 \\ 
0 & 0 & 1%
\end{bmatrix}%
\in C_{G}(a_{2})$, $[x,a_{1}]=1$ so $x=%
\begin{bmatrix}
1 & 0 & x_{13} \\ 
0 & 1 & x_{23} \\ 
0 & 0 & 1%
\end{bmatrix}%
\in C_{G}(a_{1})$ and $[y,x]=%
\begin{bmatrix}
1 & 0 & y_{12}x_{23} \\ 
0 & 1 & 0 \\ 
0 & 0 & 1%
\end{bmatrix}%
=1=%
\begin{bmatrix}
1 & 0 & 0 \\ 
0 & 1 & 0 \\ 
0 & 0 & 1%
\end{bmatrix}%
$. Then $y_{12}x_{23}=0$. But if $y_{12}\neq 0$, then $x_{23}=0$ otherwise
(1.) is contradicted while if $x_{23}\neq 0$, then $y_{12}=0$ otherwise (2.)
is contradicted. It follows that either $[y,a_{1}]=1$ or $[a_{2},x]=1$ so $%
\tau $ holds in $G$. For a fixed representation $G\leq _{H}UT_{3}(R)$
satisfying \ the Lame Property is a sufficient condition for $\tau $ to hold
in $G$; however, it is not a necessary condition for $G$ to satisfy $\tau $.
(None the less we shall subsequently see that having at least one
representation satisfying the Lame Property is necessary and sufficient.)

 
 The result (proven in \cite{GS4}) that every $3$-generator model of $%
\mathcal{Q}(H)\cup \text{diag}(H)$ is already a model of the $Th_{\forall }(G)$
provides a treasure trove of counterexamples.

 
 Let $G$ be a model of $\mathcal{Q}(H)\cup \text{diag}(H)$ and let $R$ be a
locally residually-$\mathbb{Z}$ ring. \ We say that $R$ is \emph{appropriate for $G$} provided

\begin{enumerate}\item
$G\leq _{H}UT_{3}(R)$ \ and
\item
$R$ is generated by the entries of the elements of $G$.
\end{enumerate}
\smallskip

 
 Now $R=\mathbb{Z}\times \mathbb{Z=Z}e_{1}~+~\mathbb{Z}e_{2}$ where $%
e_{1}=(1,0)$ and $e_{2}=(0,1)$ is residually-$\mathbb{Z}$ . Consider the $3$%
-generator subgroup $G\leq _{H}UT_{3}(R)$ generated by%
\begin{equation*}
a_{1}=%
\begin{bmatrix}
1 & 0 & 0 \\ 
0 & 1 & 1 \\ 
0 & 0 & 1%
\end{bmatrix}%
~,~~a_{2}=%
\begin{bmatrix}
1 & 1 & 0 \\ 
0 & 1 & 0 \\ 
0 & 0 & 1%
\end{bmatrix}%
~~\text{and }b=%
\begin{bmatrix}
1 & 0 & 0 \\ 
0 & 1 & e_{1} \\ 
0 & 0 & 1%
\end{bmatrix}%
.
\end{equation*}%
Since $1-e_{1}=e_{2}$ , $R=\mathbb{Z}\times \mathbb{Z}$ is generated by the
entries of the elements of $G$. Since $e_{1}e_{2}=0$, $e_{1}$ is a zero
divisor in $R$.

 
 Now $b\in C_{G}(a_{1})\backslash Z(G)$ and $b_{23}=e_{1}$ is a zero
divisor in $R$. Hence, the representation violates the Lame Property. Every $%
3$-generator model of $\mathcal{Q}(H)\cup \text{diag}(H)$ is already a model of $%
Th_{\forall }(H)$. So this $G=\langle a_{1},a_{2},b\rangle $ satisfies $\tau 
$. Now this $G$ is obtained from $H$ by extending $C_{H}(a_{1})$ introducing
a new parameter. (It is a rank $1$ centralizer extension relative to the
category $\boldsymbol{N}_{2}$ of $2$-nilpotent groups.)

 
 Let $\theta $ be an indeterminate over $\mathbb{Z}$ . Then the
polynomial ring $\mathbb{Z}[\theta ]$ is residually-$\mathbb{Z}$ and we
could have just as well embedded this $G$ into $UT_{3}($ $\mathbb{Z}[\theta
] $ ) as the subgroup generated by%
\begin{equation*}
a_{1}=%
\begin{bmatrix}
1 & 0 & 0 \\ 
0 & 1 & 1 \\ 
0 & 0 & 1%
\end{bmatrix}%
~,~~a_{2}=%
\begin{bmatrix}
1 & 1 & 0 \\ 
0 & 1 & 0 \\ 
0 & 0 & 1%
\end{bmatrix}%
~~\text{and }b=%
\begin{bmatrix}
1 & 0 & 0 \\ 
0 & 1 & \theta \\ 
0 & 0 & 1%
\end{bmatrix}%
.
\end{equation*}%
Since $\theta $ is an entry of an element of $G$, $\mathbb{Z}[\theta ]$ is
also appropriate for $G$. Since $\mathbb{Z}[\theta ]$ is an integral domain,
the representation does satisfy the Lame Property.

 
 Anticipating an application to be used later in this paper, suppose $%
G_{0}$ is a model of $\mathcal{Q}(H)\cup \text{diag}(H)$ and let $a_{i}$ be a free
generator of $H$, $i\in \{1,2\}$. Suppose $G$ is obtained from $G_{0}$ by
extending $C_{G_{0}}(a_{i})$. That is%
\begin{equation*}
G=\langle G_{0},t;~rel(G_{0}),~[t,C_{G_{0}}(a_{i})]=1\rangle _{\mathcal{N}%
_{2}}.
\end{equation*}%
Using a big powers argument, we get a discriminating family of retractions $%
G\rightarrow G_{0}$ via%
\begin{equation*}
\left\{ 
\begin{array}{c}
g\mapsto g~,~g\in G_{0} \\ 
t\mapsto a_{i}^{n}~,~n\in \mathbb{Z}%
\end{array}%
\right. .
\end{equation*}%
It follows that $G$ is universally equivalent to $G_{0}$.

 
 Now suppose we fix a model $G_{0}$ of $\mathcal{Q}(H)\cup \text{diag}(H)\cup
\{\tau \}$ which admits a representation $G_{0}\leq _{H}UT_{3}(R)$
satisfying the Lame Property. Let $(g_{1},g_{2})\in G_{0}^{2}$ and suppose
the system%
\begin{equation*}
S\left\{ 
\begin{array}{c}
\lbrack a_{2},y]=1 \\ 
\lbrack y,a_{1}]=[g_{2},g_{1}]%
\end{array}%
\right.
\end{equation*}%
has no solution in $G_{0}$. Let $[g_{2},g_{1}]=z=%
\begin{bmatrix}
1 & 0 & z_{13} \\ 
0 & 1 & 0 \\ 
0 & 0 & 1%
\end{bmatrix}%
$. Let $Y$ be the element $%
\begin{bmatrix}
1 & z_{13} & 0 \\ 
0 & 1 & 0 \\ 
0 & 0 & 1%
\end{bmatrix}%
$ of $UT_{3}(R)$. Then $[a_{2},Y]=1$ and $[Y,a_{1}]=[g_{2},g_{1}]$ so $%
Y\notin G_{0}$.

 
 Let $G_{1}$ be the subgroup $\langle G_{0},Y\rangle $ of $UT_{3}(R)$.
Collecting and simplifying we see a typical element of $G_{1}$ has the form $%
uY^{n}~[Y,w]$ where $n\in \mathbb{Z}$ and $(u,w)\in G_{0}^{2}$. The matrix
representing this element has the form $%
\begin{bmatrix}
1 & u_{12}+nz_{13} & \ast \\ 
0 & 1 & u_{23} \\ 
0 & 0 & 1%
\end{bmatrix}%
$. 
Now suppose \[C=
\begin{bmatrix}
1 & c_{12}+nz_{13} & \ast \\ 
0 & 1 & 0 \\ 
0 & 0 & 1
\end{bmatrix}
\in C_{G_{1}}(a_{2})\setminus Z(G_2)\ \ \text{and} \ \ 
B=
\begin{bmatrix}
1 & 0 & \ast \\ 
0 & 1 & b_{23} \\ 
0 & 0 & 1
\end{bmatrix}
\in C_{G_{1}}(a_{1})\setminus Z(G_1).\]
 Assume further that $[C,B]=1$. Then $%
\begin{bmatrix}
1 & 0 & (c_{12}+nz_{13})b_{23} \\ 
0 & 1 & 0 \\ 
0 & 0 & 1%
\end{bmatrix}%
=%
\begin{bmatrix}
1 & 0 & 0 \\ 
0 & 1 & 0 \\ 
0 & 0 & 1%
\end{bmatrix}%
$ so \ $(c_{12}+nz_{13})b_{23}=0$. Now, $\ c_{12}+nz_{13}\neq 0$ and $%
b_{23}\neq 0$. Moreover, $B$ has the form $u[Y,w]$ with $u\in
C_{G_{0}}(a_{1})$ looking like $%
\begin{bmatrix}
1 & 0 & \ast \\ 
0 & 1 & b_{23} \\ 
0 & 0 & 1%
\end{bmatrix}%
$. \ Since $b_{23}$ is a zero divisor in $R$ that contradicts \ that the
representation $G_{0}$ satisfies the Lame Property. Hence,
if $C\in C_{G_1}(a_2), B\in C_{G_1}(a_1)$ and $[C,B]=1$, then  either $c_{12}+nz_{13}=0$ or $b_{23}=0$ so either $[C,a_{1}]=1$ or $[a_{2},B]=1$ and $G_{1}$ satisfies $\tau $.

 Similarly, if the system%
\begin{equation*}
T\left\{ 
\begin{array}{c}
\lbrack x,a_{1}]=1 \\ 
\lbrack a_{2},x]=[g_{2},g_{1}]%
\end{array}%
\right.
\end{equation*}%
has no solution in $G_{0}$ we can extend to a model of $\mathcal{Q}(H)\cup
\text{diag}(H)\cup \{\tau \}$ in which $T$ has a solution.

 
 Getting back to $G_{1}$, suppose the system%
\begin{equation*}
T\left\{ 
\begin{array}{c}
\lbrack x,a_{1}]=1 \\ 
\lbrack a_{2},x]=[g_{2},g_{1}]%
\end{array}%
\right.
\end{equation*}%
has no solution in $G_{1}$. Suppose further that $G_{1}$ admits a
representation $G_{1}\leq _{H}UT_{3}(R)$ satisfying the Lame Property. Then
we can extend $G_{1}$ to a model $G$ of $\mathcal{Q}(H)\cup \text{diag}(H)\cup
\{\tau \}$ in which $T$ has a solution. We have the chain $G_{0}\leq
G_{1}\leq G$. $G$ is a model of $\mathcal{Q}(H)\cup \text{diag}(H)\cup \{\tau \}$
in which the system%
\begin{equation*}
S\cup T\left\{ 
\begin{array}{c}
\begin{array}{l}
\lbrack a_{2},y]=1 \\ 
\lbrack y,a_{1}]=[g_{2},g_{1}]%
\end{array}
\\ 
\begin{array}{l}
\lbrack x,a_{1}]=1 \\ 
\lbrack a_{2},x]=[g_{2},g_{1}]%
\end{array}%
\end{array}%
\right.
\end{equation*}%
has a solution. So, if every model $G$ of $\mathcal{Q}(H)\cup \text{diag}(H)\cup
\{\tau \}$ has at least one representation satisfying the Lame Property,
then $G$ embeds in $UT_{3}(R)$ where $R$ is an integral domain and we are
finished. To prove that every model of $\mathcal{Q}(H)\cup \text{diag}(H)\cup
\{\tau \}$ admits a representation satisfying the Lame Property, we need
some results from model theory.

 
 Let $\mathbb{M}$ be the model class operator. We first paraphrase a
result from Bell and Slomson \cite{BS}.

\smallskip

\begin{prop}
Let $L$ be a first order language with equality. Let $K$ be the class of all 
$L$structures and let $X\subseteq K$. Then there is a set $S$ of sentences
of $L$ such that $X=\mathbb{M}(S)$ if and only if $X$ is closed under
isomorphism and ultraproducts and $K\backslash X$ is closed under
ultraprowers.
\end{prop}

\begin{rem}
The proof in \cite{BS} needed the Generalized Continuum Hypothesis. In view of
Shelah's \cite{S} improvement of Keisler's ultrapower theorem that hypothesis may
be omitted.
\end{rem}

\smallskip

The next result may be found in Hodges [H].

\smallskip

\begin{prop}
Let $L$ be a first order language with equality. Let $T$ be a set of
sentences of $L$. Let $T_{\forall }$ be the set of all universal sentences
of $L$ which are logical consequences of $T$. Then $\mathbb{M}(T_{\forall })$
consists of all $L$-substructures of models of $T$.
\end{prop}

\smallskip

Using Proposition 1, a striaghtforward but tedious verification reveals that
the class of all models of $\mathcal{Q}(H)\cup \text{diag}(H)\cup \{\tau \}$ \
which admit a representation satisfying the Lame Property is first order.
Moreover, since this class is closed under $H$-subgroups, it has (as an
application of Proposition 2) a set $\Phi $ of universal axioms in $L_{0}[H]$%
. Now $\mathbb{M}(\mathcal{Q}(H)\cup \text{diag}(H)\cup \{\tau \})\supseteq ~%
\mathbb{M}(\Phi )$. We would be finished if we could prove equality. It will
suffice to establish the result for finitely generated models. We shall
prove the result more generally for models $G$ of $\mathcal{Q}(H)\cup
\text{diag}(H)\cup \{\tau \}$ $\ $such that the quotient $\overline{G}=G/Z(G)$ is
finitely generated.

\smallskip

\begin{thm}\label{thm1}
Every model of $\mathcal{Q}(H)\cup \text{diag}(H)\cup \{\tau \}$ \ admits a
representation satisfying the Lame Property.
\end{thm}

\smallskip

\begin{cor}
$\mathcal{Q}(H)\cup \text{diag}(H)\cup \{\tau \}$ $\ $is an axiomatization for $Th_{\forall }(H)$.
\end{cor}

\begin{cor}
$\mathcal{Q}(H)\cup \text{diag}(H)\cup \{NZCT\}$ is an axiomatization for $Th_{\forall }(H)$.
\end{cor}

\smallskip

\begin{proof}[Proof of Theorem~\ref{thm1}] 
Let $G$ be a model of $\mathcal{Q}(H)\cup \text{diag}(H)\cup \{\tau \}$. We may assume without loss of generality
that the quotient $\overline{G}=G/Z(G)$ is finitely generated. Let $%
\overline{C_{i}}=C_{G}(a_{i})~/Z(G),~i=1,2$. Since $\langle \overline{a_{i}}~\rangle =\langle a_{i}Z(G)\rangle \subseteq \overline{C_{i}}$, $\overline{C_{i}}$ has finite rank at least $1$, $i=1,2$.

 
 Thus, $\text{rank}(\overline{C_{1}})+\text{rank}(\overline{C_{2}})\geq 2$.
Define the \emph{$C$-rank} of $G$ to be%
\begin{equation*}
\text{rank}(\overline{C_{1}})+\text{rank}(\overline{C_{2}})-1.
\end{equation*}
The proof will proceed by induction on the $C$-rank.

 
 Suppose first that the $C$-rank of $G$ is $1$. That forces
\begin{equation*}
\text{rank}(\overline{C_{1}})=1=\text{rank}(\overline{C_{2}})
\end{equation*}
and $\overline{C_{i}}~=~\langle \overline{a_{i}}~\rangle =\langle
a_{i}Z(G)\rangle $, $i=1,2$. Let $G\leq _{H}UT_{3}(R)$ be any representation
where $R$ is locally residually-$\mathbb{Z}$ .

 
 It follows from the above that every element of $\
C_{G}(a_{2})\backslash Z(G)$ looks like $%
\begin{bmatrix}
1 & m & \ast \\ 
0 & 1 & 0 \\ 
0 & 0 & 1%
\end{bmatrix}%
$ where $m\in \mathbb{Z}\backslash \{0\}$ and every element of $%
C_{G}(a_{1})\backslash Z(G)$ looks like $%
\begin{bmatrix}
1 & 0 & \ast \\ 
0 & 1 & n \\ 
0 & 0 & 1%
\end{bmatrix}%
$ where $n\in \mathbb{Z}\backslash \{0\}$.

 
 Now, for each $k\in \mathbb{Z}\backslash \{0\}$, the quasi-identity $%
\forall x~((kx=0)\rightarrow (x=0))$ holds in $\mathbb{Z}$ . Hence, these
quasi-identities hold in $R$ and consequently every representation of $G$
satisfies the Lame Property. The initial step of the induction has been
established.

 
 Now suppose $G$ has $C$-rank~$n>1$ and the result has been established
for models with $C$-rank $k$ with $1\leq k<n$.

 
 Now let $G\leq _{H}UT_{3}(R)$ be any representation of $G$ where $R$
is locally residually-$\mathbb{Z}$ . Let us extend $G$ to $\widehat{G}$ by
adjoining the elements $%
\begin{bmatrix}
1 & 0 & r \\ 
0 & 1 & 0 \\ 
0 & 0 & 1%
\end{bmatrix}%
$ as $r$ varies over $R$. Since $\tau $ depends on the $(1,2)$ and $(2,3)$
entries only and since $\widehat{G}$ has the same $C$-rank as $G$, we may
replace $G$ with $\widehat{G}$. This causes no harm since universal
sentences are preserved in subgroups; so $G$ will be a model of $\Phi $
whenever $\widehat{G}$ is.

 
 Since the $C$-rank of $\widehat{G}$ is greater than $1$ at least one
of 
\begin{equation*}
\overline{C_{i}}=C\widehat{_{G}}(a_{i})/Z(\widehat{G})~,~i=1,2
\end{equation*}
must have rank at least $2$. We may assume $\text{rank}(\overline{C_{2}})=m\geq 2$.
Choose elements $a_{2},b_{1},...,b_{m-1}~\in C_{\widehat{G}}(a_{2})$ which
project modulo $Z(\widehat{G})$ to a basis for $\overline{C_{2}}$ . Now let $%
N$ be the subgroup $\langle b_{m-1}\rangle \cdot Z(\widehat{G})$ of $%
\widehat{G}$ . Since $[\widehat{G},\widehat{G}]\leq Z(\widehat{G})\leq N$, $%
N $ is normal in $\widehat{G}$ and $\widehat{G}/N$ is abelian. Let $T$ be a
transversal for $N$ in $\widehat{G}$ so that every element $g$ is uniquely
expresses in the form $t(g)\upsilon (g)$ where $t(g)\in T$ and $\upsilon
(g)\in N$. Assume further that, for all $(p,q,z)\in \mathbb{Z}^{2}\times Z(%
\widehat{G})$, $~t(a_{1}^{p}a_{2}^{q}z)=a_{1}^{p}a_{2}^{q}$ and that $t(x)=1$
for all $x\in N$.

 
 Let $G_{0}$ be the subgroup of $\widehat{G}$ generated by the coset
representatives of $N$ in $\widehat{G}$ together with $Z(\widehat{G})$ .
Since $b_{m-1}$ is killed off the $C$-rank of $G_{0}$ is less than that of $%
\widehat{G}$ and so, by inductive hypothesis, $G_{0}$ is a model of $\Phi $.

 
 Examining the general form of a matrix in $\widehat{G}$ and setting
that guy equal to the identity matrix, we see that, modulo the law $%
[x_{1},x_{2},x_{3}]=1$, the only relations in $\widehat{G}$ are consequences
of the relations in $G_{0}$ and the relations $[C_{G_{0}}(a_{2}),b_{m-1}]=1$%
. That is, $\widehat{G}$ is the free rank $1$ extension of $C_{G_{0}}(a_{2})$
relative to $\boldsymbol{N}_{2}$ and hence $G_{0}$ and $\widehat{G}$ are
universally equivalent. Then $\widehat{G}$ is a model of $\Phi $ since $%
G_{0} $ is. Since $G\leq \widehat{G}$ and universal sentences are preserved
in subgroups, $G$ is a model of $\Phi $. That completes the induction and
proves the theorem. \end{proof}

\medskip

\section{The Theory in the Base Language}

\medskip

 
 We wish to ponder whether or not $Q^{0}(H)\cup \{NZCT\}$ axiomatizes $%
Th_{\forall }^{0}(H)$. To that end let $G_{0}$ be a model of $Q^{0}(H)\cup
\{NZCT\}$. We may assume $G_{0}$ is finitely generated. Suppose that $G_{0}$
is abelian. Then, since models of $Q^{0}(H)$ are torsion free, $G_{0}$ is
free abelian of finite rank $r$. Choose a positive integer $n$ such that $%
\left( 
\begin{array}{c}
n \\ 
2%
\end{array}%
\right) \geq $ ~max$\{r,2\}$. If $G=F_{n}(\mathcal{N}_{2})$ then $[G,G]$ is
free abelian of rank $\left( 
\begin{array}{c}
n \\ 
2%
\end{array}%
\right) $. It follows that $G_{0}$ embeds in $G$ . So every universal
sentence of $L_{0}$ true in $G$ must also be true in $G_{0}$. But $G$ is
universally equivalent to $H$. Therefore $G_{0}$ is a model of $Th_{\forall
}^{0}(H)$. So it now suffices to assume $G_{0}$ is a finitely generated
nonabelian model of $Q^{0}(H)\cup \{NZCT\}$. A consequence of a result of Gr%
\"{a}tzer and Lasker \cite{GL} is that the quasivariety generated by $H$ consists
of all groups isomorphic to a subgroup of a direct product of a family of
ultrapowers of $H$. View $H$ as $UT_{3}(\mathbb{Z})$ and taking
corresponding direct product of ultrapowers of $\mathbb{Z}$ we get a ring $R$
such that $G_{0}$ embeds in $UT_{3}(R)$. Since quasi-identities are
preserved in direct products and ultrapowers, $R$ is a model of the
quasi-identities true in $\mathbb{Z}$. That is $R$ is locally residually-$%
\mathbb{Z}$. Further we may take $R$ to be generated by the entries of a
fixed finite set of generators for $G_{0}$. Thus, $R$ may be taken finitely
generated. Therefore $R$ is residually-$\mathbb{Z}$ and so separated by the
family of retractions $R\rightarrow \mathbb{Z}$. Let $\overline{G}~$be the
subgroup $\langle G_{0},H\rangle $ of $UT_{3}(R)$. \ The retractions $%
R\rightarrow \mathbb{Z}$ induce group retractions $\overline{G}\rightarrow H$
and these separate $\overline{G}$. It follows that $\overline{G}$ is a model
of $Q(H)\cup \text{diag}(H)$. Let us keep this $\overline{G}~$in $\ $mind as we
move on.

 
 Let us say a group is \emph{$(G_{0},H)$-group} if it contains a
distinguished copy of each of $G_{0}$ and $H$. The meanings of \emph{$(G_{0},H)$-subgroup} and \emph{$(G_{0},H)$-homomorphism} are readily
apparent. A \emph{$(G_{0},H)$-ideal} is the kernel of a $(G_{0},H)$-homomorphism. Equivalently, a $(G_{0},H)$-ideal in a $(G_{0},H)$-group $G$
is a subgroup $K$ normal in $G$ such that $K\cap G_{0}=\{1\}=K\cap H$.

 
 Suppose $I$ is a nonempty index set and $(G_{i})_{i\in I}$ \ is a
family of $(G_{0},H)$-groups indexed by $I$. Let $G$ be the direct product $%
\dprod\limits_{i\in I}~G_{i}$. We have the diagonal embeddings%
\begin{eqnarray*}
\alpha &:&G_{0}\rightarrow G \\
\alpha (g)(i) &=&g~~\text{for all }g\in G_{0}~,~i\in I
\end{eqnarray*}%
and%
\begin{eqnarray*}
\beta &:&H\rightarrow G \\
\beta (h)(i) &=&h~\text{ for all }h\in H~,~i\in I.
\end{eqnarray*}%
We view $G$ as a $(G_{0},H)$-group using these diagonal embeddings. Let $%
\dcoprod $ be the $\mathcal{N}_{2}$ free product (See \cite{N}). Let $%
G_{1},~G_{2},~$and $G$ be $2$-nilpotent groups. Let $\varphi
_{i}:~G_{i}\rightarrow G$ be a homomorphism\ $\ i-1,2$. Then there is a
unique homomorphism $\varphi :G_{1}\dcoprod G_{2}\rightarrow G$ such that 
\begin{equation*}
\varphi \mid _{G_{i}}=\varphi _{i}~,~i=1,2\text{.}
\end{equation*}%
Getting back to $\overline{G}$ and letting $\Gamma =G_{0}\dcoprod H$ we see
there is a unique $(G_{0},H)$-homomorphism $\varphi :\Gamma \rightarrow 
\overline{G}$ . Note $\varphi $ is surjective since $G_{0}$ and $H$ generate 
$\overline{G}$ . If $K=Ker(\varphi )$, then $K$ is a $(G_{0},H)$-ideal in $%
\Gamma $ such that $\Gamma /K$ is a model of $Q(H)$. Hence the set $\mathbb{K%
}$ of all $(G_{0},H)$-ideals, $K$, in $\Gamma $ such that $\Gamma /K$ is a
model of $Q(H)$ is nonempty. Let $K_{0}=\dbigcap\limits_{K\in \mathbb{K}}~K$%
. Let $U_{H}(G_{0})=\Gamma /K_{0}$.

 
 We call $U_{H}(G_{0})$ the \emph{universal $H$-extension of $%
G_{0}$}. We claim that $U_{H}(G_{0})$ is a $(G_{0},H)$-group which is a
model of $Q(H)$ and that if $G$ is any $(G_{0},H)$-group which is a model of 
$Q(H),$ then there is a unique $(G_{0},H)$-homomorphism $U_{H}(G_{0})%
\rightarrow G$. Put another way, we claim that $U_{H}(G_{0})$ is an initial
object in the category whose objects are $(G_{0},H)$-groups which are models
of $Q(H)$ and whose morphisms are $(G_{0},H)$-homomorphisms.

 
 From $K\cap G_{0}=\{1\}=K\cap H$ for all $K\in \mathbb{K}$ we get $%
K_{0}\cap G_{0}=1=K_{0}\cap H$ as $K_{0}\leq K$ for all $K\in \mathbb{K}$ .
Thus, $K_{0}$ is a $(G_{0},H)$-ideal and $U_{H}(G_{0})$ is a $(G_{0},H)$%
-group.

 
 We get a $(G_{0},H)$-homomorphism%
\begin{equation*}
\varphi :\Gamma \rightarrow \dprod\limits_{K\in \mathbb{K}}~(\Gamma /K)~~%
\text{via}
\end{equation*}%
$\varphi (\gamma )=(\gamma K)_{K\in \mathbb{K}}$ \ \ for all $\gamma \in
\Gamma $. $K_{0}=~Ker(\varphi )$. It follows that $U_{H}(G_{0})$ isomorphic
to the image of $\varphi $. Since quasi-identities of $L_{0}[H]$ are
preserved in direct products and $H$-subgroups, $U_{H}(G_{0})$ is a model of 
$Q(H)$.

 
 Now suppose $G$ is any $(G_{0},H)$-group which is a model of $Q(H)$.
Since $\Gamma =G_{0}\dcoprod H$ we get a unique homomorphism $\pi :\Gamma
\rightarrow G$ which restricts to the identity on each of $G_{0}$ and $H$.
Then $Ker(\pi )\in \mathbb{K}$ and $\pi $ induces $\ \overline{\pi }%
:U_{H}(G_{0})\rightarrow G$. Since every homomorphism is determined by its
effect on a generating set and $\Gamma =G_{0}\dcoprod H$, $\overline{\pi }$
is unique. Our claims have been established.


\begin{qu}
\label{qn1}
Is NZCT preserved in $U_{H}(G_{0})$?
\end{qu}


 
 If so, $U_{H}(G_{0})$ would be a model of $Th_{\forall }(H)$ and
hence a model of $Th_{\forall }^{0}(H)$. Since universal sentences are
preserved in subgroups, we would have $G_{0}$ a model of $Th_{\forall
}^{0}(H)$.

 
 We observe that 
\begin{equation*}
G_{0}\cap Z(U_{H}(G_{0}))\leq Z(G_{0}).
\end{equation*}%
We further observe that $Z(G_{0})$ coinciding with $G_{0}\cap
Z(U_{H}(G_{0})) $ is a necessary condition for Question~\ref{qn1} to have a positive
answer. Recall we are taking $G_{0}$ nonabelian. Let $g_{1}$ and $g_{2}$ \
be noncommuting elements of $G_{0}$ and $g\in Z(G_{0})\backslash (G_{0}\cap
Z(U_{H}(G_{0})))$, then since $g_{1}$ and $g_{2}$ each commute with $g\notin
Z(U_{H}(G_{0}))$ , \ \smallskip NZCT would be violated.


\begin{qu}\label{qn2}
Is $Z(G_{0})=G_{0}\cap Z(U_{H}(G_{0}))$?
\end{qu}


\section{Questions}


Let $c\geq 2$ be an integer. Let $r=\max \{2,c-1\}$. Let $s$ be
any integer such that $s\geq r$. Let $G=F_{s}(\mathcal{N}_{c})$. Then $%
F_{\omega }(\mathcal{N}_{c})$ is discriminated by the family of retractions $%
F_{\omega }(\mathcal{N}_{c})\rightarrow G$. It follows that the $F_{n}(%
\mathcal{N}_{c})$ have the same universal theory relative to $L_{0}$ for all 
$n\geq r$ and that the $F_{n}(\mathcal{N}_{c})$ have the same universal
theory relative to $L_{0}[G]$ for all $n\geq s$ (See [GS 1]). We let $S_{c}$
be the set of all universal sentences of $L_{0}$ true in $F_{n}(\mathcal{N}%
_{c})$. Since $F_{n}(\mathcal{N}_{c})\leq F_{r}(\mathcal{N}_{c})$ for all $%
0\leq n\leq r$ and universal sentences are preserved in subgroups, $S_{c}$
is actually the set of all universal sentences of $L_{0}$ true in every free 
$c$-nilpotent group.

 
 For each integer $n\geq 0$ we define $CT(n)$ to be the following
universal sentence of $L_{0}.$%
\begin{eqnarray*}
\forall x_{1},x_{2},x_{3},w_{1},...,w_{n}~((([w_{1},...,w_{n},x_{2}] &\neq
&1)\wedge (([x_{1},x_{2}]=1)\wedge ([x_{2},x_{3}]=1)) \\
&\rightarrow &([x_{1},x_{3}]=1)).
\end{eqnarray*}%
The interpretation of $CT(n)$ in any group $G$ is that the relation of
commutativity is transitive on $G\backslash Z_{n}(G)$ where $Z_{n}(G)$ is
the $n$-th term of the upper central series of $G$.

 
 Equivalently, it asserts that the centralizer of any element $g\in
G\backslash Z_{n}(G)$ is abelian. It was shown in \cite{FGMRS} that the free $c$%
-nilpotent groups satisfy $CT(c-1)$. Questions~\ref{qn3}--\ref{qn6} 
 below are
also due to A.G. Myasnikov.


\begin{qu}\label{qn3}
Does\textbf{\ }$Q(H)\cup \{NZCT\}$ axiomatize $%
S_{2}=Th_{\forall }^{0}(H)$?
\end{qu}


A positive answer to Question~\ref{qn1} would imply a positive answer to Question~\ref{qn3}.


\begin{qu}\label{qn4}
Let $s\geq 2$ and let $G=F_{s}(\mathcal{N}_{2})$. Does $%
Q(G)\cup \text{diag}(G)\cup \{NZCT\}$ axiomatize $Th_{\forall }(G)$?
\end{qu}


More generally -


\begin{qu}\label{qn5}
Let $c\geq 2$ and $r=\max \{2,c-1\}$. Let $G=F_{n}(%
\mathcal{N}_{c})$. Does $Q^{0}(G)\cup \{CT(c-1)\}$ axiomatize $S_{c}$?
\end{qu}


\begin{qu}\label{qn6}
Let $c\geq 2$ and $s\geq r=\max \{2,c-1\}$. Let $%
G=F_{s}(\mathcal{N}_{c})$. Does $Q(G)\cup diag(G)\cup \{CT(c-1)\}$
axiomatize $Th_{\forall }(G)$?
\end{qu}


\begin{qu}\label{qn7}
Let $\theta $ be an indeterminate over $\mathbb{Z}$ .
Must every finitely generated model of $Th_{\forall }(H)$ $H$-embed in $%
UT_{3}(\mathbb{Z}[\theta ])$?
\end{qu}


\bibliographystyle{plain}
\bibliography{references}

\end{document}